\documentclass[12pt]{article}
\usepackage{amsmath,amssymb,amsfonts,graphicx,wrapfig,amsthm}

\newtheorem{theorem}{Theorem}
\newtheorem{lemma}{Lemma}

\newcommand{\T}{\mathcal{T}}
\newcommand{\C}{\mathcal{C}}

\newcommand{\A}{\mathcal{A}}

\newcommand{\F}{\mathcal{F}}
\newcommand{\Z}{\mathbb{Z}}

\newcommand{\E}{\mathbb E}
\newcommand{\N}{\mathbb N}

\title{Logical convergence laws via stochastic approximation and Markov processes} 

\author{Yury Malyshkin\footnote{Tver State University, Russia; yury.malyshkin@mail.ru}, Maksim Zhukovskii\footnote{The University of Sheffield, UK; m.zhukovskii@sheffield.ac.uk}}

\date{}

\begin{document}

\maketitle

\begin{abstract}
Since the paper of Kleinberg and Kleinberg, SODA'05, where it was proven that the preferential attachment random graph with degeneracy at least 3 does not obey the first order 0-1 law, no general methods were developed to study logical limit laws for recursive random graph models with arbitrary degeneracy. Even in the (possibly) simplest case of the uniform attachment, it is still not known whether the first order convergence law holds in this model. We prove that the uniform attachment random graph with bounded degrees obeys the first order convergence law. To prove the law, we describe dynamics of first order equivalence classes of the random graph using Markov chains. The convergence law follows from the existence of a limit distribution of the considered Markov chain. To show the latter convergence, we use stochastic approximation processes.
\end{abstract}

\section{Introduction}
\label{sec:Intro}

We consider first-order (FO) sentences about graphs in the language containing the adjacency $\sim$ and the equality $=$ relations.
For example, the sentence 
$$
\forall x\forall y \, (x=y)\vee (x\sim y) \vee (\exists z \, (z\sim x) \vee (z\sim y))
$$
describes the property of having diameter at most 2. 

A {\it random graph} $G_n$ on the vertex set $[n]:=\{1,\ldots,n\}$ is a random element of the set of all (simple) graphs on $[n]$. It was proven by Glebskii, Kogan, Liogon'kii and Talanov~\cite{GKLT69} and independently by Fagin~\cite{F76} that, for every FO sentence $\phi$ either asymptotically almost all graphs on $[n]$ satisfy $\phi$, or asymptotically almost all graphs on $[n]$ do not satisfy $\phi$. In other words, letting $G_n$ be uniformly distributed, we get that either $\mathbb{P}(G_n\models\phi)\to 1$, or $\mathbb{P}(G_n\models\phi)\to 0$ as $n\to\infty$. Roughly speaking, the descriptive power of FO logic is weak in the sense that it does not express properties that are not trivial on typical large enough graphs. This phenomena is known as the {\it zero-one law}. More generally, a sequence of random graphs $G_n$, $n\in\mathbb{N}$, {\it obeys the FO 0-1 law}, if, for any FO sentence $\phi$, $\lim_{n\to\infty}\mathbb{P}(G_n\models\phi)\in\{0,1\}$. The comparison of descriptive powers of logics via the validity of limit laws for these logics appears to be helpful, e.g., for the induced subgraph isomorphism problem. It can be used to prove that the minimum quantifier depth $q$ of a sentence that expresses the property of containing an induced subgraph isomorphic to a given graph $F$ is at least $\frac{|E(F)|}{|V(F)|}+2$~\cite{VZ}, while the truth value of a FO sentence of quantifier depth $q$ can be determined on an $n$-vertex graph in time $O(n^q)$, see~\cite[Proposition 6.6]{Libkin}. 

The most studied model in the context of FO 0-1 laws is the {\it binomial random graph} $G(n,p)$ (see, e.g., \cite{Janson,Survey,Spencer_Ehren}), where every edge is drawn independently with probability $p$. In particular $G(n,1/2)$ is just a graph chosen uniformly at random. The above mentioned classical 0-1 law (for $p=1/2$) is generalised to all $p=p(n)$ such that $\min\{p,1-p\}n^{\alpha}\to\infty$ for every $\alpha>0$, see~\cite{Spencer_Ehren} (in particular, this is true for all constant $p\in(0,1)$).  On the other hand, the FO 0-1 law fails for $G(n,p=n^{-\alpha})$, where $\alpha\in(0,1)$ is rational \cite{Shelah}. Moreover, even the FO convergence law fails for this random graph (a sequence of random graphs $G_n$, $n\in\mathbb{N}$, obeys {\it the FO convergence law}, if, for every FO sentence $\phi$, $\lim_{n\to\infty}\mathbb{P}(G_n\models\phi)$ exists). Many other models are studied in the context of logical laws: random regular graphs \cite{Haber}, random geometric graphs \cite{M99}, uniform random trees \cite{McColm}, etc. (see, e.g., \cite{Muller,Strange,Zhuk_Svesh,Winkler}). However, the usual combinatorial tools that are applied to prove logical laws seem to be insufficient to study the logical behaviour of attachment models that are, in particular, used to model real networks (see, e.g., \cite{Hof16}). 

In~\cite{Kleinberg} Kleinberg and Kleinberg observed that the classical Bollob\'{a}s--Riordan preferential attachment random graph with degeneracy at least 3 does not obey the FO 0-1 law.  Since that, there was no significant progress in the study of logical limit laws for attachment models. In particular, it is still unknown whether the classical preferential attachment random graph obeys the FO convergence law. Though the study of random graphs is dominated by the binomial random graph (and a similar uniform model), properties of preferential attachment models better resemble those of real-world networks such as the graph of the Web, social networks, and citation networks. Let us recall that the preferential attachment graphs were introduced by Barab\'{a}si and Albert~\cite{BA_PA} and later were formalised by Bollob\'{a}s and Riordan~\cite{BR_PA}. Nevertheless, for the sake of simplicity of computations, in this paper we study a close but simpler model --- uniform attachment~\cite{FP-GPR,JKMS}.\\

The attachment models are built recursively, on each step one new vertex is added to the graph, from which $m$ new edges are drawn randomly to the old vertices. The most studied attachment models are uniform and preferential attachment.  In the uniform attachment model probabilities to draw an edge to a vertex are the same for all vertices, while in the preferential attachment model probabilities are proportional to the degrees of the respective vertices. In the context of FO limit laws, the following is known:
\begin{itemize}
\item the FO 0-1 law holds for the tree models (when only one edge is drawn at each step, i.e. $m=1$ \cite{MZ21}), for both preferential and uniform attachment,
\item for the non-tree uniform model (when we draw $m\geq 2$ edges at each step) and the preferential attachment model with the degeneracy at least 3 ($m\geq 3$ edges are drawn at each step) there is no 0-1 law~\cite{Kleinberg,MZ21},
\item the $\mathrm{FO}^{m-2}$ convergence law is known to be true for the uniform attachment~\cite{MZ22} ($\mathrm{FO}^{\gamma}$ is the fragment of FO logic comprising all sentences with at most $\gamma$ variables),  the $\mathrm{FO}^{m-3}$ convergence law holds true for some variations of the preferential attachment \cite{M22}.
\end{itemize}
Thus, for the entire FO logic, we only know that the FO 0-1 law fails if $m$ is large enough ($m\geq 2$ for the uniform attachment and $m\geq 3$ for the preferential attachment), while it is still unclear whether the FO convergence law fails at least for some $m$. Constructions of sentences with non-trivial limit probabilities are quite straightforward: since, for $m\geq 3$, the expected number of cliques of size $m+1$ converges to a finite limit, a sentence saying that there exist at least $K$ cliques of size $m+1$ is bounded away both from 0 and 1, for $K$ large enough (see details in~\cite{Kleinberg,MZ21}). Though for the existential fragment of the FO logic, the convergence law clearly holds for any attachment model (it immediately follows from the definition of the model), no approach to study the validity of the convergence law for the entire FO logic has been developed. In this paper, we develop a method to prove FO convergence laws for attachment models, and apply it to the uniform attachment with bounded degrees.\\


The main tool for proving logical laws is {\it the Ehrenfeucht-Fra\"{\i}ss\'{e} game} (see, e.g., \cite[Chapter 11.2]{Libkin}). Let us recall the rules of the game. The board consists of two vertex--disjoint graphs $G$ and $H$. There are two players, {\it Spoiler} and {\it Duplicator}. In each round, Spoiler chooses a vertex either in $G$, or in $H$; then Duplicator chooses a vertex in the remaining graph. Duplicator's objective is to show that the graphs are similar by establishing a partial isomorphism between $G$ and $H$ determined by the chosen vertices. When Duplicator fails, she immediately looses. The Ehrenfeucht-Fra\"{\i}ss\'{e} game provides a connection between the existence of the winning strategy of Duplicator in the game in $R$ rounds on two graphs and their indistinguishability in terms of FO sentences with quantifier depth at most $R$. This connection could be formulated in the following way. Let us recall that {\it the quantifier depth} of a FO sentence is, roughly, the maximum number of nested quantifiers in this sentence (for a formal definition, see \cite[Definition 3.8]{Libkin}).
\begin{theorem}
\label{thm:Ehren}
Duplicator wins the game on graphs $G$ and $H$ in $R$ rounds if and only if, for every FO sentence $\phi$ with quantifier depth at most $R$, either $\phi$ is true on both $G$ and $H$ or it is false on both graphs.
\end{theorem}

The direct consequence of Theorem~\ref{thm:Ehren} is that if for every $\varepsilon>0$ and $R\in\mathbb{N}$ there exist graph families $\A_i$, $i\in[M]$ ($M$ does not depend on $n$), such that, for any two representatives of one family, Duplicator wins the game in $R$ rounds (which is equivalent to indistinguishability in FO logic with quantifier depth at most $R$) and
$$\mathbb{P} \left(G_n\in\A_i\right)\to p_i,\,\, i\in [M],\quad \sum_{i=1}^{M}p_i>1-\varepsilon,$$
then $G_n$ satisfies the FO convergence law. 

Let us now describe the main difficulty in the application of Theorem~\ref{thm:Ehren} to proving convergence laws for uniform attachment random graphs. Assume that the two players play the $R$-rounds game on two large uniform attachment random graphs $G_1\subset G_2$ on $[n_1]$ and $[n_2]$ respectively. Let $r$ be large enough (depending on $R$). For an induced subgraph $F$ of a graph $G$, we call the induced subgraph of $G$ containing all vertices that are at distance at most $r$ from some vertex of $F$ {\it the $r$-neighbourhood of $F$}. In particular, the induced subgraph spanned by all vertices that are at distance at most $r$ from a given vertex $v$ is {\it the $r$-neighbourhood of $v$}. It can be shown that there exists $n_0$ such that very likely $G'_1:=G_1\backslash[n_0]$ and $G'_2:=G_2\backslash[n_0]$ are locally almost trees --- every vertex has the $r$-neighbourhood containing at most one cycle (we call a connected graph with exactly one cycle {\it unicyclic}). Moreover, for every admissible rooted tree $T$ of depth $r$, there are many vertices such that their $r$-neighbourhoods are isomorphic to $T$ in both graphs $G'_1,G'_2$. Then, for Duplicator to win it is enough to guarantee that, for each {\it $\mathrm{FO}_R$-equivalence class} $\C$ (two graphs are {\it $\mathrm{FO}_R$-equivalent} if they are not distinguishable by a sentence in the considered fragment of FO logic --- with quantifier depth at most $R$) 
and for every $a\in[3,r]$, the numbers of $r$-neighbourhoods of $a$-cycles $C_a$ that have isomorphic representatives in $\mathcal{C}$ are either equal in $G'_1,G'_2$ or large in both graphs. On the one hand, it is not difficult to give a structural description of logical equivalence classes of unicyclic graphs. On the other hand, there are equivalence classes with infinitely many admissible unicyclic graphs, and this makes the analysis of dynamics of the distribution of numbers of $r$-neighbourhoods that belong to a given class hard. However, if we bound the degrees of $G_n$, then this is no longer the case --- the number of representatives in each of the equivalence classes becomes bounded as well. We shall finally note that subtrees ``growing from the kernel $[n_0]$'' in $G_1,G_2$ also require a special treatment, and thus, having no restrictions on degrees, we face an obstacle similar to those as we overviewed for unicyclic subgraphs that do not overlap with $[n_0]$.\\

Let us introduce the model of graphs $G_n=G_n(m,d)$ that we consider in the paper. We start with a complete graph $G_m$ on $m$ vertices. Then on each step, we construct a graph $G_n$ by adding to $G_{n-1}$ a new vertex and drawing $m$ edges from it to different vertices, chosen uniformly among vertices with degrees less than $d$. Note that for such a procedure to be possible, we need the condition $d\geq 2m$. The case $d=2m$ is easier since in this case all but a constant number of vertices have degree $d$. It has been already considered separately in~\cite{Marxiv} and requires a different approach that cannot be generalised to other $d$. 

Note that $G_n$, for $m\geq 2$, still does not obey the FO 0-1 law --- the reason is the same as for the original uniform attachment model, see \cite{MZ21}. Indeed, if we consider the number of {\it diamond graphs} (for $m=2$) or the number of complete graphs on $m+1$ vertices (for $m\geq 3$), it could be proven (similar to the way it was done in Section~2 of \cite{MZ21}, but with modifications based on the arguments that appear in Sections~\ref{sec:4},~\ref{sec:5} of the present paper) that the probability to have a certain amount of such graphs is bounded away from both $0$ and $1$.

Let us formulate our main result.
\begin{theorem}
\label{th:main}
For every $m\geq 2$ and $d>2m$, $G_n(m,d)$ obeys the $\mathrm{FO}$ convergence law.
\end{theorem}

To prove Theorem~\ref{th:main} we show the convergence of cardinalities of the above mentioned logical equivalence classes of $r$-neighbourhoods using an approximation by Markov chains and applying some results about the existence of the limit distribution of a Markov chain. To build such an approximation we use so-called stochastic approximation processes (see, e.g., \cite{Chen03,Pem07} for more details). To our knowledge, an application of approximation processes to prove logical limit laws is novel. We hope that it may be used to prove the FO convergence for the original uniform attachment model and for some other recursive models as well.\\

The structure of the paper is as follows. 
 In Section~\ref{sec:2} we provide the basic results for a stochastic approximation that we would use in Sections~\ref{sec:3} and~\ref{sec:6}. The winning strategy of Duplicator is described in Section~\ref{sec:9}. It is based on the graph structure induced by the constantly many initial vertices (that we describe in Section~\ref{sec:4}) and the distribution of cardinalities of above mentioned equivalence classes of unicyclic graphs which is studied in Section~\ref{sec:8}. For the latter, we need to investigate separately the numbers of rooted trees (Section~\ref{sec:6}) and cycles (see Section~\ref{sec:5} for upper bounds and Section~\ref{sec:7} for lower bounds). Finally, in Section~\ref{sec:3} we prove auxiliary results about the asymptotic behaviour of the number of vertices of a given degree.

 \section{Stochastic approximation processes}
\label{sec:2}

Let us consider an $r$-dimensional {\it stochastic approximation process} $Z(n)$ (see \cite{Chen03} for more details on stochastic approximation processes) with corresponding filtration $\F_n$ which is defined in the following way





\begin{equation}
Z(n+1)-Z(n)=\frac{1}{n+1}\left(F(Z(n))+E_{n+1}+R_{n+1}\right),
\end{equation}
where $E_n$, $R_n$ and the function $F$ satisfy the following conditions. There exists $U\subset\mathbb{R}^r$ such that $Z_n\in U$ for all $n$ almost surely (a.s. for brevity) and
\begin{itemize}
\item[${\sf A1}$]
The function $F:\mathbb{R}^r\to\mathbb{R}^r$ is continuous and bounded in some neighbourhood of $U$, has a unique root $\theta$ in $U$, such that in some neighbourhood (in the entire $\mathbb{R}^r$) of the root
$$F(x)=H(x-\theta)+O\left(|x-\theta|^{a}\right)$$
for some $a>1$, where the matrix $H$ is stable, i.e. the real parts of the eigenvalues of $H$ are strictly negative. The largest of them denoted by $-L$ satisfies $L>1/2$.

\item[${\sf A2}$] For any $\epsilon>0$
$$\sup_{|x-\theta|>\epsilon,x\in U}F^t(x)(x-\theta)<0.$$
Note that for condition ${\sf A2}$ to hold it is enough for the derivative matrix of $F(x)$ to exist and to be stable in $U$.

\item[${\sf A3}$] $E_n$ is a martingale difference with respect to $\F_n$ (recall that a process $E_n$ is a {\it martingale difference} process with respect to a filtration $\F_n$ if $\E(E_{n+1}|\F_n)=0$ for all $n$) and for some $\delta\in(0,1/2)$, $R_n=O(n^{\delta})$ a.s. (i.e. there exists a non-random constant $C$, such that $\limsup_{n\to\infty}\frac{|R_n|}{n^{\delta}}\leq C$ a.s.), and
$$\sum_{n=1}^{\infty}\frac{E_{n+1}}{n^{1-\delta}}<\infty\quad\text{a.s.}$$
Note that due to convergence theorems for martingale differences (see \cite[Appendix~B]{Chen03}) for the last condition to hold it is enough that $\sup_{n}\E(|E_{n+1}|^2|\F_n)<\infty$ a.s.
\end{itemize}
We need the following result (see the proof in \cite[Theorem~3.1.1]{Chen03}):
\begin{theorem}
Under the above conditions, $Z(n)\to\theta$ a.s. with the convergence rate
$$|Z(n)-\theta|=o(n^{-\delta})\quad\text{a.s.} \left(\text{i.e. } \frac{|Z(n)-\theta|}{n^{-\delta}}\to 0 \text{ a.s.}\right).$$
\label{th:law}
\end{theorem}

\section{Number of vertices of fixed degree}
\label{sec:3}

In our model at step $n+1$ the probability to draw an edge to a given vertex equals to 
\begin{equation}
\label{eq:add_vertex}
1-\frac{{n-N_d(n)-1\choose m}}{{n-N_d(n)\choose m}}=1-\frac{n-N_d(n)-m}{n-N_d(n)}=\frac{m}{n-N_d(n)},
\end{equation}
where $N_k(n)$ is the number of vertices with degree $k$ at time $n$ for $k\in[m,d]$. In order to use this formula, we study an asymptotical behaviour of $N_k(n)$. Let $X_k(n):=N_k(n)/n$, $m\leq k\leq d$.
Let us consider the equation
$$\left(\frac{m}{m+1-x}\right)^{d-m}=x.$$
Since  $d>2m$, it has a unique root $x=\rho_d$ in $(0,2m/d)$. Let us define
\begin{equation}
\label{eq:roots}
\rho_k:=\frac{(1-\rho_d)m^{k-m}}{(m+1-\rho_d)^{k-m+1}},\quad k=m,\ldots,d-1.
\end{equation}

\begin{lemma}
\label{lem:conv_rate}
$X_k(n)\to \rho_k$ with rate $|X_k(n)-\rho_k|=o(n^{-1/2+\delta})$ for any $\delta>0$ a.s.
\end{lemma}

\begin{proof}

Let $\mathcal{F}_n$ be the filtration that corresponds to the graphs $G_n$.
We get
\begin{align*}
\E\left(N_m(n+1)-N_m(n)|\F_n\right) & =1-\frac{m}{n-N_d(n)}N_m(n),\\
\E\left(N_k(n+1)-N_k(n)|\F_n\right) & =\frac{m}{n-N_d(n)}\left(N_{k-1}(n)-N_k(n)\right),\quad k=m+1,\ldots,d-1,\\
\E\left(N_d(n+1)-N_d(n)|\F_n\right) & =\frac{m}{n-N_d(n)}N_{d-1}(n).
\end{align*}
Since  the total number of vertices of degree $d$ does not exceed $\frac{2mn}{d}$, if $d>2m$ we get that $X_{d}(n)\leq \frac{2m}{d}<1$. Note that for $X_k(n)$ we get
\begin{equation}
\E\left(X_{k}(n+1)-X_k(n)|\F_n\right)=\frac{1}{n+1}\left(\E\left(N_k(n+1)-N_k(n)|\F_n\right)-X_k(n)\right).
\end{equation}
Hence, if we define functions (on $[0,1)^{d-m}\times[0,\frac{2m}{d})$)
\begin{align}
\label{eq:func}
f_m(x_m,\ldots,x_d)&=1-\left(\frac{m}{1-x_d}+1\right)x_m,\notag\\
f_k(x_m,\ldots,x_d)&=\frac{m}{1-x_d}x_{k-1}-\left(\frac{m}{1-x_d}+1\right)x_k,\quad k=m+1,\ldots,d-1,\\
f_d(x_m,\ldots,x_d)&=\frac{m}{1-x_d}x_{d-1}-x_d,\notag
\end{align}
we would get that for all $k\in[m,d]$,
\begin{equation}
\E\left(X_{k}(n+1)-X_k(n)|\F_n\right)=\frac{1}{n+1}f_k\left(X_m(n),\ldots,X_d(n)\right).
\end{equation}
For the vector $Z(n):=(X_{m}(n),\ldots,X_d(n))$ we have the following representation
$$Z(n+1)-Z(n)=\frac{1}{n+1}\left(F(Z(n))+(n+1)(Z(n+1)-\E(Z(n+1)|\F_n))\right),$$
where $F(x_m,\ldots,x_d)=(f_m(x_m,\ldots,x_d),\ldots,f_d(x_m,\ldots,x_d))^t$. Set $$E_{n+1}=(n+1)(Z(n+1)-\E(Z(n+1)|\F_n)),\quad R_{n+1}=0.$$
Let us find nulls of the system $F(x_m,\ldots,x_d)=0$, i.e. the system
\begin{equation}
\left\{
\begin{array}{cccc}
\label{eq:system}
1-\left(\frac{m}{1-x_d}\right)x_m &=& x_m, & \\
\frac{m}{1-x_d}(x_{k-1}-x_k) &=& x_k,&\quad k=m+1,\ldots,d-1,\\
\frac{m}{1-x_d}x_{d-1} &=& x_d. &
\end{array}
\right.
\end{equation}
We get
\begin{align*}
x_m &=\frac{1-x_d}{m+1-x_d},\\
x_k & =\frac{m}{m+1-x_d}x_{k-1},\quad k=m+1,\ldots,d-1.
\end{align*}
Hence for $k=m+1,\ldots,d-1$
$$x_k=\frac{(1-x_d)m^{k-m}}{(m+1-x_d)^{k-m+1}}.$$
For $x_d$ we get that
$$\frac{m}{1-x_d}\frac{(1-x_d)m^{d-1-m}}{(m+1-x_d)^{d-1-m+1}}-x_d=0,$$
which is equivalent to
$$\left(\frac{m}{m+1-x_d}\right)^{d-m}=x_d.$$
This equation has a unique root $x_d=\rho_d$ in $(0,2m/d)$, which results in the existence of a unique solution $x_k=\rho_k$, $k=m,\ldots,d-1$.
Note that the system~\eqref{eq:system} is equivalent (by summing all rows) to the system
\begin{equation}
\left\{
\begin{array}{cccc}
1-\left(\frac{m}{1-x_d}\right)x_m &=&x_m,&\\
\frac{m}{1-x_d}(x_{k-1}-x_k) &= &x_k,&\quad k=m+1,\ldots,d-1,\\
1 &=&x_m+\ldots+x_d.&
\end{array}
\right.
\label{eq:sys}
\end{equation}
Let check the conditions of Theorem~\ref{th:law}. For non-zero partial derivatives of functions $f_k,$ $k=m,\ldots,d$, we would get:

\begin{equation}
\left\{
\begin{array}{cccc}
\frac{\partial f_m}{\partial x_m}(x_m,\ldots,x_d) &= & -\frac{m}{1-x_d}-1, &  \\
\frac{\partial f_m}{\partial x_d}(x_m,\ldots,x_d) &= & -\frac{m}{(1-x_d)^2}x_m, &  \\
\frac{\partial f_k}{\partial x_{k-1}}(x_m,\ldots,x_d)&= & \frac{m}{1-x_d}, &\quad k=m+1,\ldots,d-1, \\
\frac{\partial f_k}{\partial x_{k}}(x_m,\ldots,x_d)&= & -\frac{m}{1-x_d}-1, &\quad k=m+1,\ldots,d-1, \\
\frac{\partial f_k}{\partial x_{d}}(x_m,\ldots,x_d)&= & \frac{m}{(1-x_d)^2}(x_{k-1}-x_k), & \quad k=m+1,\ldots,d-1, \\
\frac{\partial f_d}{\partial x_{d-1}}(x_m,\ldots,x_d)&= & \frac{m}{1-x_d}, &  \\
\frac{\partial f_d}{\partial x_{d}}(x_m,\ldots,x_d)&= & -1+\frac{m}{(1-x_d)^2}x_{d-1}. & 
\end{array}
\right.
\label{eq:func_der}
\end{equation}
Hence, the characteristic polynomial of the derivative matrix is
\begin{align*}
P(\lambda) & =(-1)^{d-m}\left(\frac{m}{1-x_d}\right)^{d-m}\left(-\frac{m}{(1-x_d)^2}x_m\right)\\
&\quad+\sum_{k=m+1}^{d-1}(-1)^{d-k}\left(\frac{m}{1-x_d}\right)^{d-k}\left(-\frac{m}{1-x_d}-1-\lambda\right)^{k-m}\frac{m}{(1-x_d)^2}(x_{k-1}-x_k)\\
&\quad+\left(-\frac{m}{1-x_d}-1-\lambda\right)^{d-m}\left(-1+\frac{m}{(1-x_d)^2}x_{d-1}-\lambda\right)\\
&=(-1)^{d-m}\sum_{k=m+1}^{d}\left(\frac{m}{1-x_d}\right)^{d-k}\left(\frac{m}{1-x_d}+1+\lambda\right)^{k-m}\frac{m}{(1-x_d)^2}x_{k-1}\\
&\quad-(-1)^{d-m}\sum_{k=m}^{d-1}\left(\frac{m}{1-x_d}\right)^{d-k}\left(\frac{m}{1-x_d}+1+\lambda\right)^{k-m}\frac{m}{(1-x_d)^2}x_{k}\\
&\quad-(-1)^{d-m}\left(\frac{m}{1-x_d}+1+\lambda\right)^{d-m}\left(1+\lambda\right)\\
&=(-1)^{d-m}\sum_{k=m}^{d-1}\left(\frac{m}{1-x_d}\right)^{d-k-1}\left(\frac{m}{1-x_d}+1+\lambda\right)^{k-m}\frac{m}{(1-x_d)^2}x_{k}(1+\lambda)\\
&\quad-(-1)^{d-m}\left(\frac{m}{1-x_d}+1+\lambda\right)^{d-m}\left(1+\lambda\right).
\end{align*}
Let us denote $t=1+\lambda$, $c=\frac{m}{1-x_d}$. Then 
$$P(\lambda)=(-1)^{d-m+1}tQ(t),$$
where
$$Q(t):=(c+t)^{d-m}-\sum_{k=m}^{d-1}c^{d-k}(c+t)^{k-m}\frac{x_k}{1-x_d}.$$
For $Q(t)$ we get
\begin{align*}
Q(t) & =\sum_{i=0}^{d-m}{d-m\choose i}c^{d-m-i}t^{i}-\sum_{k=m}^{d-1}c^{d-k}\sum_{i=0}^{k-m}{k-m\choose i}c^{k-m-i}t^{i}\frac{x_k}{1-x_d}\\
&=\sum_{i=0}^{d-m}{d-m\choose i}c^{d-m-i}t^{i}-\sum_{i=0}^{d-m-1}c^{d-m-i}t^i\sum_{k=m+i}^{d-1}{k-m\choose i}\frac{x_k}{1-x_d}\\
&=t^{d-m}+\sum_{i=0}^{d-m-1}c^{d-m-i}t^{i}\left({d-m\choose i}-\sum_{k=m+i}^{d-1}{k-m\choose i}\frac{x_k}{1-x_d}\right).
\end{align*}

Note that
$${d-m\choose i}-\sum_{k=m+i}^{d-1}{k-m\choose i}\frac{x_k}{1-x_d}\geq {d-m-1\choose i}\left(1-\frac{\sum_{k=m+i}^{d-1} x_k}{1-x_d}\right).$$
Therefore, if $\sum_{k=m}^{d}x_k\leq 1$ (in particular, when $x_i=\rho_i$), $Q(t)$ has non-negative coefficients and, therefore, does not have roots with positive real parts. Note that $P(-1)=0$.
As result we get that the largest real part of eigenvalues of the derivative matrix equals $-1$ if $\sum_{k=m}^{d}x_k\leq 1$. Note that $\sum_{k=m}^{d}X_k(n)=1$. Therefore the process $Z(n)$ satisfies the conditions ${\sf A1,A2}$ of Theorem~\ref{th:law} on the set 
$$U=\left\{x_{m}+\ldots+x_{d}=1,x_k\geq 0,k=m,\ldots,d,x_{d}\leq \frac{2m}{d}\right\}.$$
To check condition ${\sf A3}$ we first recall that $R_{n+1}=0$. At each step we draw $m$ edges, so we change degrees of exactly $m$ vertices, while adding one new vertex. Hence, $|N_k(n+1)-N_k(n)|\leq m+1$ and $|X_k(n+1)-X_k(n)|\leq \frac{m+1}{n}$. Therefore, for $E_{n+1}$ we get
\begin{align*}
|E_{n+1}|
&\leq (n+1)\left(|Z(n+1)-Z(n)|+|\E(Z(n+1)-Z(n)|\F_n)|\right)\\
&\leq 2\frac{(n+1)(m+1)(d-m+1)}{n},
\end{align*}
which results in condition $A3$.
By Theorem~\ref{th:law}, we get statement of Lemma~\ref{lem:conv_rate}.
\end{proof}

\section{Probability to have degree less than $d$}
\label{sec:4}

We will need the following variant of the Chernoff bound.
\begin{lemma}
\label{lem:bern_sum}
Let $X_i$, $i\geq k$, $k>1$, be independent Bernoulli random variables with $\E X_i=\frac{p}{i}$ for some $p>0$. Let $S_n=\sum_{i=k}^{n}X_i$, $n>k$. Then for any $\delta>0$
\begin{align*}
\mathbb{P} (S_n \leq (1-\delta)p(\ln (n+1) - \ln (k))) &\leq \left(\frac{n+1}{k}\right)^{\frac{-\delta^2p}{2}},\\
\mathbb{P} (S_n \geq (1+\delta)p(\ln n - \ln (k-1)))   &\leq \left(\frac{n}{k-1}\right)^{\frac{-\delta^2p}{2+\delta}}.
\end{align*}
\end{lemma}
Let us consider the evolution of the degree of a given vertex. Fix a time $s$ and consider the vertex $s$ that appears at this time. It appears with the degree $m$. If its degree at time $t\geq s$ is less then $d$ the probability to draw an edge to it  (from the vertex $t+1$) equals to $\frac{m}{t-N_d(t)}$. Let $X_i$, $i\geq s$, be independent Bernoulli random variables with $\E X_i=\frac{m}{i}$. Then, due to Lemma~\ref{lem:bern_sum}, the probability that the vertex $s$ has degree less than $d$ at time $n > s$ does not exceed 
$$\mathbb{P}\left(\sum_{i=s}^{n-1}X_{i}\leq d-m-1\right)\leq  c\left(\frac{n}{s}\right)^{-m/2}$$
for some positive constant $c$.
By repeating this estimate to vertices that appear at the beginning of our graph process, we would get the following result.

\begin{lemma}
\label{lem:starting_degrees}
For any fixed $s$, with high probability (hereinafter we write `whp' for brevity, i.e. with probability tending to $1$ as $n\to\infty$) the degree of $s$ in $G_n$ equals to $d$. In particular, for any fixed $n_0$ and $a$ whp degrees of all vertices in $a$-neighbourhood of first $n_0$ vertices have degrees equal to $d$.
\end{lemma}

\section{Number of cycles: upper bound}
\label{sec:5}

Let us estimate the probability that a new cycle of length $r$ would be formed at time $n+1$. To form a cycle of length $r$ we have to connect a new vertex with two vertices joined by a path of length $r-2$ that are open to attachment (there are $n-X_d(n)$ such vertices). Since degrees of vertices do not exceed $d$, there are at most $d^{r-2}(n-X_d(n))$ ordered pairs of vertices that are open to attachment and joined by an $(r-2)$-path. Recall that the number of vertices of degree $d$ does not exceed $\frac{2m}{d}n$.
Hence probability to form a new cycle does not exceed $$\frac{m(m-1)d^{r-2}}{n-X_d(n)}\leq\frac{m(m-1)d^{r-2}}{\left(1-\frac{2m}{d}\right)n}.$$
Let $n> m(m-1)d^{r-1}$. Connecting in this way two vertices that are joined by an $(r-2)$-path could create at most $d^{r-3}$ new cycles, and at each step, we draw edges to $m(m-1)/2$ pairs of vertices. Therefore at each step we could create at most $d^{r-3}m(m-1)/2$ new cycles. As result, the number of $r$-cycles in $G_n$ is stochastically dominated by  $(d^{r-3}m(m-1)/2)\sum_{i=m(m-1)d^{r-1}+1}^{n}X_{i-1} + C_{r},$ where $X_i$ are independent Bernoulli random variables with parameters $p/i$, $p=\frac{m(m-1)d^{r-2}}{1-\frac{2m}{d}}$ and the constant $C_{r}$ equals to the maximal possible number of $r$-cycles on first $m(m-1)d^{r-1}+1$ vertices. Due to Lemma~\ref{lem:bern_sum} there are  constants $C,c>0$, such that
$$\mathbb{P}\left((d^{r-3}m(m-1)/2)\sum_{i=m(m-1)d^{r-1}+1}^{n}X_i>C\ln n\right)\leq cn^{-2}.$$
Hence, due to the Borel--Cantelli lemma, probability that there are more then $C\ln n$ cycles in $G_n$ for some $n>N$ tends to $0$ as $N\to\infty$, i.e. we proved the following result.
\begin{lemma}
\label{lem:number_of_cycle_upper_bound}
For any $r>2$, the number of cycles of length $r$ in $G_n$ is $O(\ln n)$ a.s.
\end{lemma}

Note that w.h.p. there are at most $C\ln n$ vertices in the $a$-neighbourhood of the union of all $r$-cycles. Therefore probability to draw an edge to this neighbourhood at time $n$ does not exceed $\frac{C\ln n}{n}$ (for some constant $C$), and to draw two edges does not exceed $\frac{C\ln^2 n}{n^2}$.
Therefore, by the Borel--Cantelli lemma, we get the following result.
\begin{lemma}
\label{lem:complex_graphs}
For any $\epsilon>0$ and $\ell$ there is $s$ such that with probability at least $1-\epsilon$ in $[n]\backslash [s]$ there are no connected subgraphs with at most $\ell$ vertices and at least 2 cycles.
\end{lemma}

\section{Number of rooted trees}
\label{sec:6}

For a rooted tree $T$, let $N_T(n)$ be the number of vertices that are roots of maximal subtrees of $G_n$ (a subtree is {\it maximal} in $G_n$ if all its non-leaf vertices are adjacent only to vertices of that tree) isomorphic to $T$. Note that the set of all isomorphism classes of rooted trees with degrees at most $d$ of a given depth is finite. We would refer to a maximal subtree of $G_n$ isomorphic to a tree $T$ from that set as {\it having the type $T$} (i.e. when we talk about the type of a tree in $G_n$ we assume it is rooted and maximal). Also, we call a tree $T$ {\it max-admissible}, if with positive probability its isomorphic copy is a maximal subtree of $G_n$ for large enough $n$. 
In the current section, we prove the following statement:
\begin{lemma}
\label{lem:trees}
For any max-admissible tree $T$ there is a constant $\rho_T\in(0,1)$, such that for any $\delta>0$
$$N_T(n)=\rho_T n +o(n^{1/2+\delta})\quad\text{a.s.}$$
In particular, for all $s\in\N$ and any admissible tree $T$ w.h.p. there are at least $s$ vertices in $G_n$ that are roots of maximal trees isomorphic to $T$. 
\end{lemma}
\begin{proof}
Let us fix $b\in\N$ and consider variables $X_T(n):=N_T(n)/n$ and vector $Z_b(n):=(X_{{T_i}}(n))$ over all max-admissible rooted trees $T_i$ of depth $b$ (there are only finitely many such trees). Note that the case $b=1$ refer to the number of stars and was already considered in Section~\ref{sec:3}. Let $b>1$. The order of the elements of $Z_b(n)$ (or, in other words, the order on the set of all max-admissible trees of depth $b$) is defined in a way such that an addition of new branches (that preserves the depth of the tree) increases the order. It could be done by induction on $b$, say, in the following way. If $T_1,T_2$ are stars (i.e. $b=1$), then $T_1\prec T_2$ if and only if $T_1$ has less leaves than $T_2$. Assume that $\prec$ on the set of all max-admissible trees of depth $b-1$ is defined. Let $s_1,s_2$ be the number of children of roots of trees $T_1,T_2$ of depth $b$ respectively. If $s_1<s_2$, then $T_1\prec T_2$. If $s_1=s_2=:s$, then let $T_j^1,\ldots,T_j^s$ be the subtrees of $T_j$ rooted at the children $v_j^1,\ldots,v_j^s$ of the root of $T_j$ comprising all descendants of these children and ordered in the decreasing order. Then $T_1\prec T_2$ if and only if $(T_1^1,\ldots,T_1^s)\prec_s(T_2^1,\ldots,T_2^s)$, where $\prec_s$ is the lexicographical order on the set of $s$-vectors of trees of depth $b-1$ induced by the order $\prec$. 

Note that
$$\E(X_T(n+1)-X_T(n)|\F_n)=\frac{1}{n+1}\left(\E(N_T(n+1)-N_T(n)|\F_n)-X_T(n)\right).$$
There are two ways to change $N_T(n)$ at time $n+1$. We could draw an edge to a maximum tree isomorphic to $T$ or we could create a new copy of $T$ rooted at $n+1$. Recall that due to equation \eqref{eq:add_vertex} for each given vertex of degree less than $d$ probability to draw an edge to it is 
$$\frac{m}{n-N_d(n) }=\frac{1}{n}\frac{m}{1-\frac{N_d(n)}{n}}.$$
In a rooted tree $T$, fix a non-leaf vertex $u$. Then the expected number (conditioned on $G_n$) of trees $T'$ in $G_{n}$ of type $T$ such that an edge is drawn from $n+1$ to a vertex $u'$ of $T'$ and there exists an isomorphism of rooted trees $T\to T'$ sending $u$ to $u'$ equals
$$C\frac{m\frac{N_T(n)}{n}}{1-\frac{N_d(n)}{n}}=C\frac{mX_T(n)}{1-X_d(n)},$$
where the constant $C=C(T,u)$ corresponds to the number of vertices that belong to the orbit of $u$ under the action of the automorphism group. For example, if we consider a rooted tree of depth $3$ with all inner vertices of degree $k$, then for the root $C=1$, for vertices at distance $1$ from the root $C=k$ and for vertices at distance $2$ from the root $C=k(k-1)$.
Recall that $X_d(n)\leq\frac{2m}{d}$, and hence, due to condition $2m<d$, $X_d(n)$ is bounded away from $1$.

The type of the maximal tree with root $n+1$ would correspond to the probability distribution induced by the numbers of maximal trees of depth $b-1$ at time $n$. It is defined by the types of trees of depth $b-1$, to which roots we draw $m$ edges from vertex $n+1$. Note that (conditional) probability to draw edges to trees that share a non-leaf vertex is $O(\frac{1}{n})$ a.s. It is also possible to draw an edge to a vertex from a neighbourhood of a cycle, but, due to Lemma~\ref{lem:number_of_cycle_upper_bound}, the probability to do so is of order $O(\frac{\ln n}{n})$ and does not affect our argument. Hence, drawing a given edge to the root of a given tree does not impact (up to $O(\frac{\ln n}{n})$ error term) probabilities to draw other edges to roots of other trees. Therefore, probability to create a tree of type $T$ in the vertex $n+1$ is polynomial of $\frac{X_{T_{i}}(n)}{1-X_d(n)}$ up to $O(\frac{\ln n}{n})$ error term, where $T_{i}$ are max-admissible trees of depth $b-1$. To change a type of a given maximal tree in $G_n$ (to another given type) of depth $b$ we need to draw an edge to one of its vertices and draw the rest of the edges to the roots of trees (such that every tree does not have a non-leaf vertex that belongs to a different tree --- probability of drawing edges to "intersecting" trees is of order $\frac{1}{n}$ and counted in the error term $O(\frac{\ln n}{n})$) of depth at most $b-2$ of given types (that depends on the type of a tree we want to obtain). Such probability is polynomial of $\frac{1}{1-X_d(n)}$ and $X_{T_{i}}(n)$ up to a term $O(\frac{\ln n}{n})$, where $T_i$ are max-admissible trees of depth $b-2$.

Therefore
$$\E(Z_b(n+1)-Z_b(n)|\F_n)=\frac{1}{n+1}\left(A_bZ_b(n)-Z_b(n)+Y_b+O\left(\frac{\ln n}{n}\right)\right)$$
where $A_b=A_b(Z_1(n),\ldots,Z_{b-2}(n))$ is a lower-triangular matrix with negative elements on the diagonal and non-negative under the diagonal and $Y_b=Y_b(Z_{b-1}(n),X_d(n))$ is a vector, such that the elements of both $A_b$ and $Y_b$ are polynomials of $\frac{1}{1-X_d(n)}$ and $X_{T_{i}}(n)$, where $T_{i}$ are trees of depth at most $b-2$ (for $A_b$) or exactly $b-1$ (for $Y_b$). Let us consider $F_b(Z_1,\ldots,Z_b):=A_bZ_b(n)-Z_b(n)+Y_b$ (note that $A_b$ and $Y_b$ are functions of $Z_1,\ldots,Z_{b-1}$ itself). Recall that $Z_1$ contains $X_d$, so $F_b$ is deterministic. We would use induction over $b$ to prove that there is a unique solution of the system $F_{i}(z_1,\ldots,z_i)=0, i=1,\ldots,b$ (in an appropriate area). We already established the existence of the unique (non-zero) root for the case $b=1$. Assume there are unique non-zero solutions $z_1^{\ast},\ldots,z_{b-1}^{\ast}$ of the systems $F_{i}(z_1,\ldots,z_i)=0, i=1,\ldots,b-1$. If we define $H_b(z_b)=F_b(z_1^{\ast},\ldots,z_{b-1}^{\ast},z_b)$, then $H_b(z_b)=0$ is a system of linear equations with the unique root $z_b^{\ast}$ since $A_b$ is lower-triangular with negative elements on the diagonal. Now let us show that all components of $z_b^{\ast}$ are positive. Recall that all elements under the diagonal of $A_b$ are non-negative and each (except first) row has at least one positive element outside the diagonal (if a tree is not the smallest possible, we could remove one vertex with its children from it to make it smaller). All components of $Y_b(z_{b-1}^{\ast},\rho_d)$ are non-negative as well. Hence it is enough to show that the first element of $Y_b$ is positive. It follows from the fact that the smallest max-admissible tree of depth $b$ (which corresponds to the first coordinate of $z_b$) could be obtained by drawing edges from a new vertex to the smallest max-admissible trees of depth $b-1$ and the first coordinate of $z^*_{b-1}$ is positive by the induction hypothesis.

Let us consider the vector $W_b(n)=(Z_1(n),\ldots,Z_{b}(n))$. We get that
$$\E(W_b(n+1)-W_b(n)|\F_n)=\frac{1}{n+1}\left((F_1,\ldots,F_b)+O\left(\frac{\ln n}{n}\right)\right).$$
The derivative matrix of function $(F_1,\ldots,F_b)(z_1,\ldots,z_b)$ is of following form. Around the diagonal, it has clusters of derivatives of $F_i$ with respect to $z_i$, which are lower-triangular (since $F_i= A_iz_i-z_i+Y_i$) for $i>1$ with diagonal elements at most $-1$. The cluster for $i=1$ was studied in Section~\ref{sec:3} (and has characteristic polynomial $P(\lambda)$ with biggest root $-1$). Since $F_i$ depends only on $z_1,\ldots,z_i$, all elements above diagonal clusters are $0$. Therefore the highest eigenvalue of the derivative matrix of $(F_1,\ldots,F_b)$ is $-1$ (for all possible values of the process). Hence $W_b(n)$ satisfies condition ${\sf A2}$ of Theorem~\ref{th:law}. Since functions $(F_1,\ldots,F_b)$ have second-order derivatives, condition ${\sf A1}$ is satisfied as well. To check condition ${\sf A3}$ note that if we take 
$$E_{n+1}=(n+1)(W_b(n+1)-\E(W_b(n+1)|\F_n)),$$
then 
\begin{align*}
R_{n+1} :&= (n+1)(W_b(n+1)-W_b(n))-(F_1,\ldots,F_b)-E_{n+1}\\
&=(n+1)\E(W_b(n+1)-W_b(n)|\F_n) - (F_1,\ldots,F_b) = O\left(\frac{\ln n}{n}\right)\quad\text{a.s.}
\end{align*}
and
$$|E_{n+1}|\leq (n+1)|W_b(n+1)-W_b(n)|+(n+1)|\E(W_b(n+1)-W_b(n)|\F_n)|\leq C$$
for some constant $C$ since the number of trees of depth $b$ that could be impacted by the vertex $n+1$ is bounded from above by a constant, which results in condition ${\sf A3}$.
Therefore, due to Theorem~\ref{th:law}  $W_b(n)$ converges a.s. to $(z_1^{\ast},\ldots,z_{b}^{\ast})$ with the rate $o(n^{-1/2+\delta})$ for any $\delta>0$ a.s.
\end{proof}

\section{Number of cycles: lower bound}
\label{sec:7}

By Lemma~\ref{lem:trees}, recall that for any max-admissible rooted tree $T$ of depth $r-1$ there exists $\rho_T>0$, such that
$$N_T(n)=\rho_T n +o(n^{2/3})\quad\text{a.s.}$$
Let $\T_{r-1}$ be the set of all max-admissible trees of depth $r-1$, such that its root and at least one vertex at distance $r-2$ from the root have degrees less than $d$. Note that this set is not empty. The probability to draw a cycle of length $r$ at step $n+1$ (subject to $G_n$) is at least (each pair of root and non-root vertices could be counted at most twice)
$$\sum_{T\in\T_{r-1}}N_T(n)\frac{m(m-1)}{2(n-X_{d}(n))^2}\geq \frac{1}{n}\sum_{T\in\T}\rho_T $$   
for all $n>N$ with probability tending to $1$ as $N\to\infty$. Therefore, the increase in the number of cycles at step $n+1$ is stochastically dominated from below by Bernoulli random variable with parameter $p_n=p/n$ for some $p>0$, and these variables are independent for different steps.
For any $n_0$ due to Lemma~\ref{lem:starting_degrees} w.h.p. all vertices in the $r$-neighbourhood of $[n_0]$ have degrees equal to $d$, and,  hence, w.h.p. if a cycle arises at step $n+1$, then it entirely belongs to $[n+1]\setminus[n_0]$. Due to Lemma~\ref{lem:bern_sum} and the domination by  Bernoulli random variables, there are constants $c,C,\delta>0$, that the number of $r$-cycles in $[n]\setminus[n_0]$ exceeds $c\ln n$ with probability at least $1-Cn^{-\delta}$. Therefore, we get the following result.

\begin{lemma}
\label{lem:cycles}
For any $s,r$ and $n_0$ w.h.p. there are at least $s$ cycles of length $r$ that are entirely in $[n]\backslash [n_0]$.
\end{lemma}

\section{Number of unicyclic graphs}
\label{sec:8}

Let us recall that a graph $U$ is {\it unicyclic} if it is connected and contains exactly one cycle. In other words, a unicyclic graph comprises a cycle (of length $\ell$) with disjoint trees growing from this cycle (we assume that all trees have the same depth $k$;  $\ell$ and $k$ are fixed for the rest of the section).  Let $U$ be a max-admissible unicyclic (maximality and max-admissibility in the case of unicyclic graphs are defined exactly in the same way as for trees) graph. We say that a maximal unicyclic subgraph of $G_n$ {\it has type $U$} if it is isomorphic to $U$. We have one specific type $U_0$ of unicyclic graphs with all non-leaf vertices having degree $d$. Let us call such unicyclic graphs {\it complete}. As above, $N_U(n)$ is the number of maximal subgraphs in $G_n$ isomorphic to $U$. Let us consider the vector $Z(n)=(N_{U_i}(n))_{i=1,\ldots,K}$, where $U_i$ are all non-complete unicycle graphs of {\it depth} $k$ (i.e. the depth of trees growing from the cycle) comprising an $\ell$-cycle, ordered from the smallest to the largest (the linear order on unicyclic graphs could be defined in the same way as on rooted trees), and $K=K(k,\ell)$ is the number of unicyclic graphs of such kind. Process $Z(n)$ takes values in $\Z_+^{K}$. Note that the complete unicyclic graph $U_0$ could only be obtain by adding a leaf (since the degree of a new vertex equals $m$) to a unique non-leaf vertex of $U_K$ with degree less than $d$. 
In this section, we prove that $Z(n)$ has a limit probability distribution.
\begin{lemma}
\label{lem:unicyc}
For any $i_1,\ldots,i_K$ there exists a constant $c=c(i_1,\ldots,i_K)$ such that
$$\mathbb{P}(N_{U_{1}}(n)=i_1,\ldots,N_{U_{K}}(n)=i_K)\to c$$
as $n\to\infty$, and $\sum_{i_1,\ldots,i_K\in\mathbb{Z}_+}c(i_1,\ldots,i_k)=1$. Moreover for any $n_0$
$$\mathbb{P}(N_{U_0}>n_0)\to 1$$
as $n\to\infty$.
\end{lemma}

\begin{proof}

For a fixed max-admissible unicyclic graph $U$, at time $n+1$ the value of $N_U$ may change due to the following reasons (similar to the changing of the number of rooted trees from the previous section). 
\begin{itemize}
\item A new graph may be created by drawing $2$ edges from the vertex $n+1$ to a single tree of a certain type (recall that by Lemma~\ref{lem:conv_rate}, the probability to draw an edge to a given vertex (subject to $G_n$) equals $\frac{1}{(1-\rho_d)n}+o(n^{-4/3})$ a.s.), and the rest of the edges to roots of ``disjoint'' (without common non-leaf vertices) trees of certain types. By Lemma~\ref{lem:trees}, for a max-admissible tree $T$, we have $N_T=\rho_T n+\theta_T(n) n^{2/3}$, where, for every $C$, $\max_{T:\,|V(T)|\leq C}\theta_T(n)\to 0$ a.s. Hence, in the same way as in the previous section, conditional probability of creating a unicyclic graph of type $U$ this way (given $G_n$) equals $\frac{c_{U}}{n}+o(n^{-4/3})$ a.s. for some constant $c_U\geq 0$.
\item A $U$-isomorphic graph may be created from a fixed smaller unicyclic subgraph $H$ of the type $U'$, if the vertex $n+1$ sends an edge to a non-leaf vertex of $H$ and the rest of the edges to roots of ``disjoint'' trees of certain fixed types in a way that $H$ becomes of type $U$ (i.e. the maximal subgraph comprising the same cycle and having the same depth as $H$ becomes of type $U$). Conditional probability of creating a maximal subgraph of type $U$ in this way (given $H$) equals $\frac{c_{U',U}}{n}+o(n^{-4/3})$ for some constant $c_{U',U}\geq 0$.
\end{itemize}
If $U$ has at least one non-leaf vertex of degree less than $d$, the previous procedure could reduce $N_U(n)$ by drawing an edge to a unicyclic graph of the type $U$. Once a maximal unicyclic subgraph becomes complete, it never changes its type.

Note that the conditional probability (given $G_n$) to perform more than one of such operations (maybe for different types of $U$) at the same time equals $O\left(\frac{1}{n^{2}}\right)$ a.s.
We prove the existence of a limit probability distribution for $Z(n)$ by considering an auxiliary process which is defined below.

Let us consider a Markov chain $S(n)=(S_1(n),\ldots,S_K(n))$ on $\Z_+^{K}$ (see, e.g., \cite[Chapter 6]{GS01} for more details on Markov chains and corresponding terminology) with transition probabilities (we denote $c_i:=c_{U_i}$, $c_{j,i}:=c_{U_{j},U_{i}}$ for brevity)
\begin{itemize}
\item for $i\in[K]$,
$$\mathbb{P}(S_i(n+1)=S_i(n)+1,\,S_j(n+1)=S_j(n),j\neq i)=\frac{c_{i}}{n};$$ 
\item for $1\leq j<i\leq K$, 
$$\mathbb{P}(S_i(n+1)=S_i(n)+1,\,S_j(n+1)=S_j(n)-1,\,S_k(n+1)=S_k(n),k\neq i,j)=\frac{c_{j,i}S_j}{n};$$
\item $c_{K,0}=1/(1-\rho_d)$ and
$$\mathbb{P}(S_K(n+1)-S_K(n)=-1,\,S_j(n+1)=S_j(n),j\neq K)=\frac{c_{K,0}S_K}{n};$$
\item 
$$
\mathbb{P}(\forall i\, S_{i}(n+1)=S_i(n))=1-\sum_{i=1}^K \frac{c_{i}}{n}-\sum_{1\leq j<i\leq K}\frac{c_{j,i}S_j}{n}-\frac{c_{K,0}S_K}{n}.
$$
\end{itemize}  

Since sums of error terms $o\left(n^{-4/3}\right)$ and $O\left(\frac{1}{n^{2}}\right)$ converge, such terms would not impact process $Z(n)$ after some random moment $N$, and hence the existence of the limit probability distribution for $Z(n)$ follows from its existence for $S(n)$ for any initial distribution.

Note that $c_{1}\neq 0$, $C_{K,0}\neq 0$ and from the definition of $c_{U,U'}$ and the ordering, it follows that
for any $i,j$ strictly between $1$ and $K$, we get that
\begin{itemize}
\item there are $1=i_1<\ldots<i_t=i$, such that $c_{i_{s},i_{s+1}}\neq 0$ for all $s\in[t-1]$,
\item there are $j=j_1<\ldots<j_p=K$, such that $c_{i_{s},i_{s+1}}\neq 0$ for all $s\in[p-1]$.
\end{itemize}
This implies that $S(n)$ is aperiodic and irreducible.
Note that $S(n)$ is not time-homogeneous. Let us consider a random walk $S'(t)$ on $\Z_+^{K}$ that reflects only those moves of $S(n)$ when it changes its state (i.e. for every $t$, $S'(t):=S(n_t)$, where $n_t$ is the $t$-th moment $n$ such that $S(n)\neq S(n-1)$). Since $\frac{c}{n}$, $c\neq 0$, forms a divergent series, by Borel--Cantelli lemma, all coordinates of $S$ change infinitely many times a.s., so $S'$ is well defined. Also, since the conditional probability (given $S(n-1)=x$) to change the state at time $n$ is $\frac{c}{n-1}$, where $c$ depends only on $x$, we get that the conditional probability that the state at time $n$ becomes $y$ (for a fixed $y\neq x$), subject to $S_n=x$ and the event that the state is changed, does not depend on $n$, and only depends on $x$ and $y$. Thus, $S'(t)$ is time-homogeneous and its transition probabilities are given by
\begin{itemize}
\item $\mathbb{P}(S'_i(t+1)=S'_i(t)+1,\,S'_j(t+1)=S'_j(t),j\neq i)=\frac{c_{i}}{D(S'(t))},$
\item $\mathbb{P}(S'_i(t+1)=S'_i(t)+1,\,S'_j(t+1)=S'_j(t)-1,\,S'_k(t+1)=S_k(t),k\neq i,j)=\frac{c_{j,i}S'_j(t)}{D(S'(t))},$
\item $\mathbb{P}(S'_K(t+1)-S'_K(t)=-1,\,S'_j(t+1)=S'_j(t),j\neq K)=\frac{c_{K,0}S'_K(t)}{D(S'(t))},$
\end{itemize}
where 
$$D(S'(t))=\sum_{i=1}^Kc_{i}+\sum_{1\leq j<i\leq K}c_{j,i}S'_j(t)+c_{K,0}S'_K(t).$$
Let us consider $S'_1(t)$. There are constants $c_{-}$ and $c_{+}$ such that 
\begin{align*}
\mathbb{P}(S'_1(t+1)-S'_1(t)= -1|S'(t)) & \geq c_{-}\frac{S'_1(t)}{|S'(t)|+1},\\
\mathbb{P}(S'_1(t+1)-S'_1(t)=1|S'(t))   & \leq c_{+}\frac{1}{|S'(t)|+1}.
\end{align*}
Hence, for large enough $S'_1(t)$ (i.e. with $S'_1(t)\geq N$ for some $N\in\N$), $$\E(S'_1(t+1)-S'_1(t)|S'(t),S'_1(t)>N,S'_1(t+1)\neq S'_1(t))<C<0$$
for some constant $C$. Therefore $S'_1(t)$ is positively persistent. Consider $W_i(t)=(S'_{1}(t),\ldots,S'_i(t))$, $i=1,\ldots,K$. Let us assume that $W_i(t)$ is positively persistent, and prove that the same is true for $W_{i+1}(t)$. Note that there are constants $C_1,C_2>0$, such that 
\begin{align*}
\mathbb{P}(S'_{i+1}(t+1)-S'_{i+1}(t)=1|S'(t)) &<C_1\frac{|W_{i+1}(t)|+1}{|S'(t)|+1},\\
\mathbb{P}(S'_{i+1}(t+1)-S'_{i+1}(t)=-1|S'(t))&>C_2\frac{S'_{i+1}(t)}{|S'(t)|+1}.
\end{align*}
Let $N>\frac{C_1}{C_2}$, $N\in\N$. We get
$$
\E\left(\left.S'_{i+1}(t+1)-S'_{i}(t)\right|S'(t),S'_{i+1}(t)>N|W_{i+1}(t)|,S'_{i+1}(t+1)\neq S'_{i+1}(t)\right)<C<0
$$ 
for some constant $C$. Hence, the probability $\mathbb{P}(S'_{i+1}(t+t')\leq N|W_{i+1}(t+t')|\,|S'_{i+1}(t)\leq N|W_{i+1}(t)|)$ is bounded away from $0$ as $t'\to\infty$. Since $W_i(t)$ is positively persistent, it implies that $W_{i+1}(t)=(W_i(t),S'_{i+1}(t))$ is positively persistent as well.

As result, for each state $s=(s_1,\ldots,s_K)$ probabilities $\mathbb{P}(S'(t+t')=s|S'(t)=s)$ (as $t'\to\infty$) are bounded away from $0$. Hence, the same is true for probabilities $\mathbb{P}(S(n_{t+t'})=s|S(n_t)=s)$ as $t'\to\infty$, and for $\mathbb{P}(S(t+t')=s|S(t)=s)$ as well.
Therefore, there exists limit distribution for $S(n)$ (and for $Z(n)$).

The second part of Lemma~\ref{lem:unicyc} follows from Lemma~\ref{lem:cycles} and the existence of limit distribution for $Z(n)$.

\end{proof}

\section{Convergence laws}
\label{sec:9}

Fix $R\in\N$ and set $a=3^R$. For $r\in\mathbb{N}$, let us call a unicyclic graph comprising a cycle of length at most $r$ and trees of depth exactly $r$ an {\it $r$-graph}. The cycle of a unicyclic graph is called its {\it kernel}. An $r$-graph is {\it complete} if all its vertices have degrees either 1 or $d$, and all its trees are perfect and of the same depth. 

Below we define graph properties ${\sf Q1}$ and ${\sf Q2}$ that imply the existence of a winning strategy of Duplicator. Consider some integer numbers  $n>N_0>n_0$. We say that a graph $G$ with maximum degree $d$ on $[n]$ has {\it the property} ${\sf Q1}$, if

\begin{enumerate}
\item any two cycles of length at most $a$ with vertices outside of $[n_0]$ are at distance at least $3a$ from each other;
\item any vertex outside of $[N_0]$ is at distance at least $3a$ from $[n_0]$;
\item any vertex from $[N_0]$ has degree $d$;
\item for any max-admissible tree $T$ of depth at most $a$, there are at least $R$ maximal subgraphs in $G$ isomorphic to $T$ at distance at least $a$ from $[N_0]$ and each other, and the same is true for any complete $a$-graph $U$.
\end{enumerate}

Now, assume that $n_1>n_2>N_0$, and $G^1,G^2$ are graphs on $[n_1]$ and $[n_2]$ respectively such that $G^1|_{[N_0]}=G^2|_{[N_0]}$. We say that the pair of graphs $(G^1,G^2)$ has {\it the property} ${\sf Q2}$, if for any non-complete max-admissible $a$-graph $U$, 
\begin{itemize}
\item either the numbers of maximal subgraphs in $G^i$ isomorphic to $U$ are equal for $i\in\{1,2\}$,
\item or, in both graphs, there are at least $R$ maximal copies of $U$ that are distance at least $a$ from $[N_0]$ and each other.
\end{itemize}
Note that, if $G^1,G^2$ have maximum degree $d$, the property ${\sf Q1}$, and the pair $(G^1,G^2)$ has the property ${\sf Q2}$, then, for any positive integer $\delta\leq 2^R$, the numbers of maximal subgraphs in $G^i$ isomorphic to a given non-complete max-admissible $a$-graph $U$ with all vertices at distance at least $\delta$ from $[n_0]$ are equal for $i\in\{1,2\}$.

\begin{lemma}
\label{lem:game}
If both graphs $G^1$, $G^2$ have maximum degree $d$, have the property ${\sf Q1}$, and the pair $(G^1,G^2)$ has the property ${\sf Q2}$, then Duplicator has a winning strategy in the Ehrenfeucht-Fra\"{\i}ss\'{e} game on graphs $G^1,G^2$ in $R$ rounds.
\end{lemma}

\begin{proof}
Let us define the winning strategy of Duplicator. For a vertex $v$ and $r\in\N$, let $B_r(v)$ be the  $r$-neighbourhood of $v$ (i.e., the closed ball in the graph metric of radius $r$ and the center at $v$). In the same way, for a set of vertices $U$, $B_r(U)=\cup_{v\in U}B_r(v)$ is the $r$-neighbourhood of $U$. Note that we omit a reference to a graph in the notation for these balls --- each time we use the notation, the host graph would be clear from the context. For every round $i\in[R]$, we denote by $x_1,\ldots,x_i$ and $y_1,\ldots,y_i$ the vertices chosen in graphs where Spoiler and Duplicator made the $i$-th move respectively (say, $G^1$ and $G^2$ respectively). For $x_j$ and $y_j$, $j\in[i]$, let us denote by $X_j$ and $Y_j$ the unions of sets of vertices of all kernels of non-complete $2^R$-graphs in $G^1$ and $G^2$ respectively such that these kernels are completely outside $[n_0]$, and are at distance at most $2^{R-j+1}$ from $x_j$ and $y_j$ respectively. Since $G^1,G^2$ have property ${\sf Q1}$, each of the sets $X_j,Y_j$ comprises at most 1 cycle. For a set $A\subset V(G^1)$ and a set $B\subset V(G^2)$, we say that they are {\it $i$-equivalent}, and write $A\equiv_i B$, if the following conditions are fulfilled:
\begin{itemize}
\item the sets of $j\in[i]$ such that the respective vertex $x_j$ ($y_j$) belongs to $A$ ($B$) are equal, 
\item there exists an isomorphism $\varphi:G^1|_A\to G^2|_B$ of the induced subgraphs on $A$ and $B$ that maps $x_j$ to $y_j$ for all $j$ such that $x_j\in A$ and preserves (in both directions) all kernels that are outside $[n_0]$ of non-complete $2^R$-graphs,
\item if $x_i$ is at distance at most $2^{R-i+1}$ from $[n_0]$, then $\varphi$ can be extended to an isomorphism of $G^1|_{A\cup[N_0]}$ and $G^2|_{B\cup[N_0]}$ that maps every vertex of $[N_0]$ to itself.
\end{itemize}

We define the strategy by induction on the number of rounds that have been just played. Fix $i\in[R]$ and assume that, in round $i$, Spoiler makes a move in $G_1$ (without loss of generality --- if the move was done in $G_2$, then the strategy is exactly the same), and that 
 for all $j\leq i-1$, $B_{2^{R-j+1}}(X_j\cup\{x_j\})\equiv_j B_{2^{R-j+1}}(Y_j\cup \{y_j\})$.
Note that, if $i=1$, then there are no additional requirements on the graphs. We also note that, due to the assumption, the map sending $x_j$ to $y_j$, $j\in[i-1]$, is an isomorphism of $G_1|_{\{x_1,\ldots,x_{i-1}\}}$ and $G_2|_{\{y_1,\ldots,y_{i-1}\}}$. So if we succeed with the induction step, then we eventually get that Duplicator wins the game.\\


\begin{enumerate}

\item If $d(x_i,[n_0])\leq 2^{R-i+1}$ then Duplicator chooses $y_i=x_i$. We need to check that $B_{2^{R-i+1}}(X_i\cup\{x_i\})\equiv_j B_{2^{R-i+1}}(Y_i\cup \{y_i\})$. Due to the property ${\sf Q1}$, every cycle of length at most $a$ which is completely outside $[N_0]$ is far from $x_i=y_i$, and so $X_i=Y_i=\varnothing$. Also, the balls $B_{2^{R-i+1}}(x_i)$ and $B_{2^{R-i+1}}(y_i)$ are equal and lie entirely in $[N_0]$. If there is $j<i$ such that $x_j\in B_{2^{R-i+1}}(x_i)$, then, by the induction hypothesis, $B_{2^{R-j+1}}(x_j)\equiv_j B_{2^{R-j+1}}(y_j)$, implying that $B_{2^{R-j+1}}(x_j)=B_{2^{R-j+1}}(y_j)$ due to the third condition in the definition of the relation $\equiv_j$. Then $x_j=y_j$ belongs to $B_{2^{R-i+1}}(x_i)=B_{2^{R-i+1}}(y_i)$. The relation $B_{2^{R-i+1}}(X_i\cup\{x_i\})\equiv_j B_{2^{R-i+1}}(Y_i\cup \{y_i\})$ follows. If there are no such $j$, then there is also no $j$ such that $y_j\in B_{2^{R-i+1}}(y_i)$, and the relation is immediate.\\ 

\item Assume that $d(x_i,[n_0])> 2^{R-i+1}$ and that there exists $j<i$ such that $x_j\in B_{2^{R-i+1}}(X_i\cup\{x_i\})$. Let $j<i$ be the biggest such round. Then $B_{2^{R-i+1}}(X_i\cup\{x_i\})\subset B_{2^{R-j+1}}(X_j\cup\{x_j\})$, and either $X_i=X_j$ or $X_i=\varnothing$. 
 Let $\mathcal{J}$ be the set of all $j'<j$ such that $x_{j'}\in B_{2^{R-j+1}}(X_j\cup\{x_j\})$. 
  Since $B_{2^{R-j+1}}(X_j\cup\{x_j\})\equiv_j B_{2^{R-j+1}}(Y_j\cup \{y_j\})$ by the induction hypothesis, we may find an isomorphism $\varphi:G^1|_{B_{2^{R-j+1}}(X_j\cup\{x_j\})}\to G^2|_{B_{2^{R-j+1}}(Y_j\cup\{y_j\})}$ such that $\varphi(x_{j'})=y_j'$ for all $j'\in\mathcal{J}$, there are no $y_{j'}\in B_{2^{R-j+1}}(Y_j\cup\{y_j\})$ for $j'\in[j-1]\setminus\mathcal{J}$, and $\varphi(X_j)=Y_j$. Duplicator chooses $y_i=\varphi(x_i)$. It is obvious that $\varphi':=\varphi|_{B_{2^{R-i+1}}(X_i\cup\{x_i\})}$ is the desired isomorphism that insures that $B_{2^{R-i+1}}(X_i\cup\{x_i\})\equiv_i B_{2^{R-i+1}}(Y_i\cup \{y_i\})$.\\ 


\item Finally, we assume that $d(x_i,[n_0])> 2^{R-i+1}$ and there are no $j<i$ such that $x_j\in B_{2^{R-i+1}}(X_i\cup\{x_i\})$. If $X_i\neq\varnothing$, then let $U_1$ be the unique maximal $2^R$-graph with the kernel $X_i$. Due to the observation after the definition of the property ${\sf Q2}$ and due to the property ${\sf Q1}$, in the other graph there exists a maximal $2^R$-graph $U_2$ isomorphic to $U_1$ such that
\begin{itemize}
\item the kernel of $U_2$ is either at least at the same distance from $[n_0]$ as the kernel of $U_1$ from $[n_0],$ or is at distance at least $2^R$ from $[n_0]$;
\item $U_2$ is at distance at least $2a$ from the cycle of length at most $2^R$ that meets the $2^{R-j+1}$-neighbourhood of $y_j$ for every $j<i$;
\item  if, for some $j<i$, the $2^{R-j+1}$-neighbourhood of $y_j$ does not meet a cycle of length at most $2^R$, then the cycle of $U_2$ is at distance greater than $2^{R-j+1}$ from $y_j$.
\end{itemize}
Consider an isomorphism $\varphi:U_1\to U_2$, that sends the vertex $x_i$ to a vertex which is at distance more than $2^{R-i+1}$ from $[n_0]$ and set $y_i=\varphi(x_i)$.  The relation $B_{2^{R-i+1}}(X_i\cup\{x_i\})\equiv_i B_{2^{R-i+1}}(Y_i\cup \{y_i\})$ is straightforward.

Finally, if $X_i=\varnothing$, then the existence of a good choice of $y_i$ follows from the property ${\sf Q1}$.(4). Indeed, there are only two options: 1) $B_{2^{R-i+1}}$ is a maximal subtree, and then there is an isomorphic maximal subtree in the other graph which is at distance at least $a$ from the neighbourhoods of all $y_j,$ $j<i$; 2) there is a complete maximal unicyclic graph comprising a cycle of length at most $R$ which is at distance at most $2^{R-i+1}$ from $x_i$, and then there is an isomorphic maximal unicyclic subgraph in the other graph which is at distance at least $a$ from the neighbourhoods of all $y_j,$ $j<i$. The choice of $y_i$ is straightforward.



\end{enumerate}


\end{proof}

Theorem~\ref{th:main} follows from the observation after Theorem~\ref{thm:Ehren} and
\begin{lemma}
\label{lem:properties}
For any $R\in\N$ and any $\varepsilon>0$ there are $N_0>n_0$, graph families $\mathcal{A}_i$, $i\in[M]$, and numbers $p_i>0$, $i\in[M]$, $\sum_{i=1}^{M}p_i>1-\varepsilon$, such  that
\begin{itemize}
\item all graphs in $\sqcup_{i\in[M]}\mathcal{A}_i$ have the property ${\sf Q1}$;
\item if $n_1>n_2>N_0$ and graphs $G^1\supset G^2$ on $[n_1]$ and $[n_2]$ respectively belong to the same family $\mathcal{A}_i$, then the pair $(G^1,G^2)$ has the property ${\sf Q2}$;
\item for every $i\in[M]$, $\lim_{n\to\infty}\mathbb{P}(G_n\in\mathcal{A}_i)=p_i$.
\end{itemize}

\end{lemma}

\begin{proof}
Fix $R\in\N$ and $\varepsilon>0$. By Lemmas \ref{lem:starting_degrees}, \ref{lem:complex_graphs}, \ref{lem:trees}, \ref{lem:unicyc}, there exist $N>N_0>n_0$ such that with probability at least $1-\varepsilon$, for all $n\geq N$, $G_n$ has maximum degree $d$ and the property ${\sf Q1}$, and all its vertices that are at distance at most $4a$ from $[N_0]$ have degree $d$. We let $\mathcal{A}$ to be the union over all $n\geq N$ of the families of graphs $G$ on $[n]$ such that the maximum degree of $G$ equals $d$, $G$ has the property ${\sf Q1}$, and all the vertices of $G$ at distance at most $4a$ from $[N_0]$ have degree $d$. It remains to partition $\mathcal{A}=\sqcup_{i=1}^M\mathcal{A}_i$ in an appropriate way.


Let $\mathcal{M}$ be all $K$-tuples (we refer to Section~\ref{sec:8} to recall the definition of $K$) of non-negative integers that are at most $R$, and set $M:=|\mathcal{M}|M_0$, where $M_0$ is the number of all admissible maximal subgraphs of $G_N$ on $[N_0]$. For each $i\in[M]$, the respective tuple $\mathbf{m}_i\in\mathcal{M}$, and the respective admissible $H$ on $[N_0]$, let $\mathcal{A}_i\subset\mathcal{A}$ be the set of all graphs $G$ from $\mathcal{A}$ such that $G|_{[N_0]}=H$ and $\tilde Z(G)=\mathbf{m}_i$, where $\tilde Z(G)$ consists of $\tilde Z_j=\min\{R,N_{U_j}\}$. By Lemma~\ref{lem:unicyc}, for every $i\in[M]$, there exists $p_i:=\lim_{n\to\infty}\mathbb{P}(G_n\in\mathcal{A}_i)$. Note that, for every graph from $\mathcal{A}$, every its maximal unicyclic subgraph with a cycle of length at most $a$ and depth $a$ that is at distance at most $a$ from $[N_0]$, is complete (since all the vertices at distance at most $4a$ from $[N_0]$ have degree $d$). The property ${\sf Q2}$ follows.


\end{proof}



%
%



\section{Acknowledgments}

The part of the study made by Y.A. Malyshkin was done in Moscow Institute of Physics and Technology and was funded by RFBR, project number 19-31-60021.


\end{document}